\documentclass[12pt]{amsart}
\usepackage{amssymb,bm,graphicx,graphics,mathrsfs,bbm,color,float}
\numberwithin{equation}{section}
\numberwithin{figure}{section}
\allowdisplaybreaks[4]
\theoremstyle{plain}
\newtheorem{theorem}{Theorem}[section]
\newtheorem{lemma}[theorem]{Lemma}
\theoremstyle{remark}
\newtheorem{remark}[theorem]{Remark}
\begin{document}
\def\O{\Omega}
\def\p{\partial}
\def\LT{{L^2(\O)}}
\def\R{\mathbb{R}}
\def\argmin{\mathop{\rm argmin}}
\def\cT{\mathcal{T}}
\def\cN{\mathcal{N}}
\def\tu{\tilde{u}}
\def\ty{\tilde{y}}
\def\d{\displaystyle}
\def\G{\Gamma}
\def\cJ{\mathcal{J}}

  \title[New Error Analysis for a Point Tracking Problem]
  {A New Error Analysis for a Point Tracking Elliptic Distributed Optimal
  Control Problem with Pointwise Control Constraints}
\thanks{This  work  was supported in part by
the National Science Foundation under
 Grant No. DMS-25-13273.}
\author{Susanne C. Brenner}
\address{Susanne C. Brenner, Department of Mathematics and Center for
Computation and Technology, Louisiana State University, Baton Rouge,
LA 70803, USA}
\email{brenner@math.lsu.edu}
\author{Li-yeng Sung}
\address{Li-yeng Sung,
 Department of Mathematics and Center for Computation and Technology,
 Louisiana State University, Baton Rouge, LA 70803, USA}
\email{sung@math.lsu.edu}
\begin{abstract}
  We develop a concise error analysis for a linear-quadratic elliptic distributed
   optimal control problem with
  point tracking  and pointwise control constraints, where  a standard finite element
  discretization and a variational discretization are treated in a unified framework.
  The analysis is self-contained modulo standard techniques in finite element
   analysis and classical
  results for elliptic partial differential equations.
\end{abstract}
%
%
\maketitle
\section{Introduction} \label{sec:Introduction}
Let $\O\subset\R^d$  $(d=2,3)$ be a bounded  convex polygonal/polyhedral domain.
Let $\beta\in(0,1]$ and $a<b$ be constants, $x_1,\ldots, x_J$ be points in $\O$, and
 $d_1,\ldots,d_J\in\R$.   The optimal control problem is to find
\begin{equation}\label{eq:OCP}
  (\bar y,\bar u)=\argmin_{(y,u)\in K}\frac12\Big(\sum_{j=1}^J [y(x_j)-d_j]^2+\beta
  \|u\|_\LT^2\Big),
\end{equation}
 where $(y,u)\in H^1_0(\O)\times \LT$ belongs to the constraint set $K$ if and only if
\begin{alignat}{3}
  \int_\O \nabla y\cdot\nabla z\,dx&=\int_\O uz\,dx&\qquad&\forall\,z\in H^1_0(\O),
  \label{eq:StateEquation}\\
  a&\leq u\leq b&\qquad&\text{almost everywhere in $\O$}.\label{eq:ControlConstraint}
\end{alignat}
\begin{remark}\label{rem:ER}
  According to elliptic regularity for convex domains
  (cf. \cite{Grisvard:1985:EPN,Dauge:1988:EBV,MR:2010:Polyhedral}), the constraint
  \eqref{eq:StateEquation} implies $y\in H^2(\O)$ and
\begin{equation}\label{eq:ERBdd}
  \|y\|_{H^2(\O)}\leq C_\O \|u\|_\LT.
\end{equation}
  Therefore the pointwise evaluation of $y$ is justified by the Sobolev embedding
   (cf. \cite{ADAMS:2003:Sobolev})
\begin{equation}\label{eq:Sobolev}
 H^2(\O)\hookrightarrow  C(\bar \O).
\end{equation}
\end{remark}
\begin{remark}\label{rem:Notation}
  We follow the standard notation for function spaces, norms and differential
  operators that can be
  found for example in \cite{Ciarlet:1978:FEM,ADAMS:2003:Sobolev,BScott:2008:FEM}.
\end{remark}
\par
 Numerical methods for this optimal control problem were investigated in the papers
 \cite{CGY:2015:Point,BDE:2016:Points,AOS:2018:Weighted,BMV:2019:Point,CL:2026:PT,LC:2026:IPHDG}
 and also in \cite[Chapter~11]{VM:2025:EOCP}.
 Our goal is to derive in a unified framework the best
 {\em a priori} error estimates in the literature that were obtained for the standard finite element
 discretization
 and  the variational discretization.  Our approach is based on
 an extension of the new error
 analysis in \cite{BSung:2025:NewAnalysis} for elliptic distributed optimal control
 problems with $L^2$ tracking and pointwise
 control constraints to  the point tracking problem \eqref{eq:OCP}--\eqref{eq:ControlConstraint}.
\par
 The rest of the paper is organized as follows.  We recall the properties of the
 continuous problem in Section~\ref{sec:Continuous} and introduce the discrete problems
  in Section~\ref{sec:Discrete}.
 We set up the convergence analysis in Section~\ref{sec:SetUp}, where two key
 terms in the error analysis
 are identified.  The error estimates are then presented in Section~\ref{sec:Errors}.
 We end with some concluding remarks in Section~\ref{sec:Conclusion}.
 Appendix~\ref{append:Interior} and Appendix~\ref{append:DG}
 provide the derivations of two technical results used in Section~\ref{sec:Errors}.
\par
 Throughout the paper we use $C$ (with or without subscripts) to denote a
 generic positive constant
 independent of the mesh size and $\beta$.
\section{Continuous Problem}\label{sec:Continuous}
 It follows from the projection theorem for  Hilbert spaces
 (cf. \cite[Theorem~5.2]{Brezis:2011:Book})
 that
  the convex minimization problem \eqref{eq:OCP}--\eqref{eq:ControlConstraint}
 has a unique solution $(\bar y,\bar u)\in K$  characterized by the
 first order optimality condition
\begin{equation}\label{eq:VI}
  \sum_{j=1}^J [\bar y(x_j)-d_j][y(x_j)-\bar y(x_j)]+\beta\int_\O
  \bar u(u-\bar u)dx\geq 0\qquad\forall\,
  (y,u)\in K.
\end{equation}
%
\par
 It is easy to check (cf. \cite[Section~2.8]{Troltzsch:2010:OC}) that
 the variational inequality \eqref{eq:VI} is equivalent to the following relations:
\begin{alignat}{3}
\int_\O \nabla\bar y\cdot\nabla z\,dx&=\int_\O \bar u z\,dx
&\qquad&\forall\,z\in H^1_0(\O),
\label{eq:State}\\
\int_\O \bar p \Delta q\,dx&=-\sum_{j=1}^J (\bar y(x_j)-d_j)q(x_j)\label{eq:Adjoint}\\
&=-\sum_{j\in\cJ}(\bar y(x_j)-d_j)q(x_j)
&\qquad&\forall\,q\in H^2(\O)\cap H^1_0(\O),\notag
\\
 \bar u&=\max(a,\min(b,-\bar p/\beta)),\label{eq:Control}
\end{alignat}
 where $\bar p\in \LT$ is the adjoint state, and
\begin{equation}\label{eq:cJ}
  \cJ=\{1\leq j\leq J: \bar y(x_j)-d_j\neq0\}.
\end{equation}
%
\subsection{The Adjoint State}\label{subsec:AdjointState}
 Let $\rho\in C^\infty_c(\O)$ such that $\rho=1$ near the tracking points in
 $\{x_j:\,j\in \cJ\}$, and
\begin{equation}\label{eq:FS}
  \Gamma(x)=\begin{cases}\d
    \frac{\ln |x|}{2\pi}&\qquad\text{$d=2$},\\[4pt]
   \d -\frac{1}{4\pi|x|}&\qquad\text{$d=3$}.
  \end{cases}
\end{equation}
\par
 Then
 we have
\begin{equation}\label{eq:FundamentalSolution}
    \Delta \big(\rho(x)\Gamma(x-x_j)\big)=\delta_j+r_j
\end{equation}
 in the sense of distributions (cf. \cite[Section~2.4]{GT:2001:EllipticPDE}),
 where $\delta_j$ is the Dirac point measure associated with $x_j$ and
 $r_j\in C^\infty_c(\O)$.
 \par
 It follows from \eqref{eq:Adjoint} and \eqref{eq:FundamentalSolution} that
\begin{align}\label{eq:AdjointDecomposition}
  \bar p
  =-\sum_{j\in \cJ} [\bar y(x_j)-d_j]\rho(x)\Gamma(x-x_j)+\tilde p,
\end{align}
 where
\begin{equation}\label{eq:tildep}
 \tilde p\in H^2(\O)\cap H^1_0(\O)
\end{equation}
 satisfies
\begin{equation*}
  \int_\O \nabla \tilde p\cdot\nabla q\,dx=\int_\O rq\,dx\qquad
  \forall\,q\in H^1_0(\O).
\end{equation*}
 Here $r=-\sum_{j\in \cJ} [\bar y(x_j)-d_j]r_j$ belongs to $C^\infty_c(\O)$.
\par\begin{remark}\label{rem:AdjointStateRegularity}
  The relations \eqref{eq:AdjointDecomposition}--\eqref{eq:tildep}
   imply that the adjoint
  state
  $\bar p$ belongs to $H^2$ away from the tracking points in $\{x_j:\,j\in\cJ\}$.
\end{remark}
\subsection{The Optimal Control}\label{subsec:OptimalControl}
 We deduce from \eqref{eq:Sobolev}, \eqref{eq:cJ},
 \eqref{eq:FS},
  \eqref{eq:AdjointDecomposition} and \eqref{eq:tildep} that
\begin{equation*}
  \lim_{x\rightarrow x_j}|\bar p(x)|=\infty \quad\text{if}\quad j\in \cJ.
\end{equation*}
 It then follows from \eqref{eq:Control} that
\begin{equation}\label{eq:OCConstant}
  \text{$\bar u$ is a constant near any tracking points in $\{x_j:\,j\in\cJ\}$.}
\end{equation}
\begin{lemma}\label{lem:OCRegularity}
  We have
\begin{equation}\label{eq:OCRegularity}
   \bar u\in W^{1,s_d}(\O),
\end{equation}
 where $s_d=6$ for $d=3$ and $s_d$ can be any number in $[0,\infty)$ for $d=2$.
\end{lemma}
\begin{proof}  In view of Remark~\ref{rem:AdjointStateRegularity} and
 \eqref{eq:OCConstant}, we can modify $\bar p$ near the tracking points in $\{x_j:\,j\in\cJ\}$
 so that
\begin{equation*}
  \bar u=\max(a,\min(b,-q/\beta))
\end{equation*}
 for a function $q\in H^2(\O)\cap H^1_0(\O)$.
   The relation \eqref{eq:OCRegularity} follows immediately
 from \cite[Lemma~7.6]{GT:2001:EllipticPDE} and the
  Sobolev embedding $H^2(\O)\hookrightarrow W^{1,s_d}(\O)$
 (cf. \cite{ADAMS:2003:Sobolev}).
\end{proof}
\begin{remark}\label{rem:StrongerResult}
  It was shown in \cite{BMV:2019:Point} that $\bar u\in W^{1,\infty}(\O)$.
  But the weaker result \eqref{eq:OCRegularity} is sufficient for the error analysis
  in Section~\ref{sec:Errors}.
\end{remark}
\subsection{KKT Conditions}\label{subsec:KKT}
 There exist $\lambda_a,\lambda_b\in\LT$ such that
\begin{align}
 &\sum_{j=1}^J[\bar y(x_j)-d_j)]y(x_j)+\beta\int_\O \bar u\,u\,dx=
 \int_\O(\lambda_a +\lambda_b)u\,dx \label{eq:KKT1}\\
 \intertext{for all $(y,u)\in H^1_0(\O)\times\LT$ that satisfy
 \eqref{eq:StateEquation},}
 & \lambda_a\geq 0 \quad\text{and}\quad \lambda_b\leq0,\label{eq:KKT2}\\
 &\int_\O \lambda_a(u-a)dx=0 \quad\text{and}\quad
 \int_\O\lambda_b(u-b)dx=0.\label{eq:KKT3}
\end{align}
\par
 Indeed, it follows from \eqref{eq:Adjoint} that
\begin{equation*}
  \sum_{j=1}^J[\bar y(x_j)-d_j]y(x_j)+\beta\int_\O \bar u\,u\,dx
  =\int_\O (\bar p+\beta \bar u)u\,dx
\end{equation*}
 for all $(y,u)\in H^1_0(\O)\times\LT$ that satisfy
 \eqref{eq:StateEquation}.
 Therefore \eqref{eq:KKT1}
 and \eqref{eq:KKT2} hold for
\begin{equation}\label{eq:Multiplier}
  \lambda_a=\max(0,\bar p+\beta\bar u) \quad\text{and}\quad
  \lambda_b=\min(0,\bar p+\beta\bar u),
\end{equation}
 and then the complementarity conditions in \eqref{eq:KKT3}
 follow from \eqref{eq:Control}.
\par
 We conclude from Remark~\ref{rem:AdjointStateRegularity},
 \eqref{eq:OCRegularity} and \eqref{eq:Multiplier}
 that
\begin{equation}\label{eq:MultiplierRegularity}
  \text{$\lambda_a$ and $\lambda_b$ belong to $W^{1,s_d}$ away
  from the tracking points in $\{x_j:\,j\in\cJ\}$.}
\end{equation}
\section{Discrete Problems}\label{sec:Discrete}
 Let $\cT_h$ be a quasi-uniform simplicial triangulation of $\O$ and
 $V_h\subset H^1_0(\O)$ be the $P_1$ finite element
 space associated with $\cT_h$.
 We will approximate the optimal state by functions in $V_h$.
 For the approximation of
 the optimal control, we have two choices.
\subsection{A Finite Element Discretization}\label{subsec:FEM}
 Let $W_h\subset\LT$ be the $P_0$ finite element space associated with $\cT_h$.
 The discrete problem is to find
\begin{equation}\label{eq:DOCP}
  (\bar y_h,\bar u_h)=\argmin_{(y_h,u_h)\in K_h}\frac12
  \Big(\sum_{j=1}^J [y_h(x_j)-d_j]^2+\beta
  \|u_h\|_\LT^2\Big),
\end{equation}
 where $(y_h,u_h)\in V_h\times W_h$ belongs to the discrete constraint
  set $K_h$ if and only if
\begin{alignat}{3}
  \int_\O \nabla y_h\cdot\nabla z_h dx&=\int_\O u_hz_hdx
  &\qquad&\forall\,z_h\in V_h,\label{eq:DStateEq}\\
  a&\leq Q_h u_h\leq b &\qquad&\text{almost everywhere in $\O$}.
  \label{eq:DControlConstraint}
\end{alignat}
 Here $Q_h$ is the orthogonal projection from $\LT$ onto $W_h$.
\par
As in the case of the continuous problem, the discrete
 problem \eqref{eq:DOCP}--\eqref{eq:DControlConstraint} has a unique
 solution characterized by the
 first order optimality condition
\begin{equation}\label{eq:DVI}
  \sum_{j=1}^J[\bar y_h(x_j)-d_j][y_h(x_j)-\bar y_h(x_j)]
  +\beta\int_\O \bar u_h(u_h-\bar u_h)dx
  \geq0 \qquad\forall\,(y_h,u_h)\in K_h.
\end{equation}
\par
 Again the discrete variational inequality \eqref{eq:DVI} is equivalent
 to the following relations:
\begin{alignat}{3}
\int_\O \nabla\bar y_h\cdot\nabla z_hdx&=\int_\O \bar u_h z_hdx
&\qquad&\forall\,z_h\in V_h,
\label{eq:DState}\\
\int_\O \nabla\bar p_h\cdot \nabla q_h\,dx&=\sum_{j=1}^J (\bar y_h(x_j)-d_j)q_h(x_j)
&\qquad&\forall\,q_h\in V_h,\label{eq:DAdjoint}\\
\intertext{where $\bar p_h\in V_h$, and}
 \bar u_h&=\max(a,\min(b,-Q_h(\bar p_h/\beta)).\label{eq:DControl}
\end{alignat}
\subsection{A Variational Discretization}\label{subsec:Variational}
Following \cite{Hinze:2005:Control}, we replace $W_h$ by $W=\LT$ in
\eqref{eq:DOCP}--\eqref{eq:DControlConstraint},
where $Q_h:\LT\longrightarrow W$ is now the identity operator.
The relations \eqref{eq:DVI}--\eqref{eq:DControl} remain valid.
\section{Set-Up for the Convergence Analysis}\label{sec:SetUp}
 We follow the approach in \cite{BSung:2025:NewAnalysis}.
  Let $(\ty_h,\tu_h)\in K_h$ be defined by
 \begin{equation}\label{eq:tuh}
   \tu_h=Q_h\bar u
 \end{equation}
  and
\begin{equation}\label{eq:tyh}
  \int_\O \nabla\ty_h\cdot\nabla z_h dx=\int_\O \tu_h z_hdx
  \qquad\forall\,z_h\in V_h,
\end{equation}
  and $\tilde p_h\in V_h$ be defined by
\begin{equation}\label{eq:tildeph}
  \int_\O \nabla \tilde p_h\cdot\nabla q_h\,dx=\sum_{j\in\cJ}(\bar y(x_j)-d_j)q_h(x_j)
  \qquad\forall\,q_h\in V_h.
\end{equation}
\begin{remark}\label{rem:Qhbaru}
  For the variational discretization in Section~\ref{subsec:Variational}, we have $\tu_h=\bar u$
  because $Q_h$ is the identity operator.
  For the finite element discretization in Section~\ref{subsec:FEM},
  the relation \eqref{eq:OCConstant} implies that $\tu_h=Q_h\bar u=\bar u$ in fixed
   neighborhoods of the tracking points in $\{x_j:j\in\cJ\}$, provided that $h$ is sufficiently small.
\end{remark}
\par
 We obtain, by \eqref{eq:State}, \eqref{eq:tyh} and \eqref{eq:tildeph},
\begin{align}\label{eq:Representation}
  \sum_{j\in\cJ}[\bar y(x_j)-d_j][\ty_h(x_j)-\bar y_h(x_j)]
  =\int_\O \nabla \tilde p_h\cdot\nabla(\ty_h-\bar y_h)dx
  =\int_\O (\tu_h-\bar u_h)\tilde p_hdx.
\end{align}
\par
 Note that, for the discretization in Section~\ref{subsec:FEM}, we have
\begin{equation}\label{eq:QhBdd}
  \|\bar u-\tu_h\|_\LT\leq C_\sharp h
\end{equation}
 by \eqref{eq:OCRegularity}, \eqref{eq:tuh} and a standard estimate for $Q_h$
 (cf. \cite{Ciarlet:1978:FEM,BScott:2008:FEM}), and $C_\sharp=0$  for the
 variational discretization
  in Section~\ref{subsec:Variational}
 because in this case $\tu_h=Q_h\bar u=\bar u$.
\par
We can write
\begin{align}\label{eq:Splitting1}
  &\sum_{j=1}^J[\bar y(x_j)-\bar y_h(x_j)]^2+\beta\|\bar u-\bar u_h\|_\LT^2\notag\\
  &\hspace{30pt}=\sum_{j=1}^J[\bar y(x_j)-\bar y_h(x_j)][\bar y(x_j)-\ty_h(x_j)]
  +\beta\int_\O (\bar u-\bar u_h)(\bar u-\tu_h)dx\\
  &\hspace{50pt}+\sum_{j=1}^J[\bar y(x_j)-\bar y_h(x_j)][\ty_h(x_j)-\bar y_h(x_j)]
  +\beta\int_\O (\bar u-\bar u_h)(\tu_h-\bar u_h)dx,
  \notag
\end{align}
 and it follows from \eqref{eq:DVI} and \eqref{eq:Representation} that
\begin{align}\label{eq:Splitting2}
  &\sum_{j=1}^J[\bar y(x_j)-\bar y_h(x_j)][\ty_h(x_j)-\bar y_h(x_j)]
  +\beta\int_\O (\bar u-\bar u_h)(\tu_h-\bar u_h)dx\notag\\
  &\hspace{30pt}
  =\sum_{j=1}^J\bar y(x_j)[\ty_h(x_j)-\bar y_h(x_j)]
  +\beta\int_\O \bar u(\tu_h-\bar u_h)dx\notag\\
  &\hspace{50pt}-\sum_{j=1}^J\bar y_h(x_j)[\ty_h(x_j)-\bar y_h(x_j)]
  -\beta\int_\O \bar u_h(\tu_h-\bar u_h)dx\\
  &\hspace{30pt}
  \leq \sum_{j\in\cJ}[\bar y(x_j)-d_j][\ty_h(x_j)-\bar y_h(x_j)]
  +\beta\int_\O \bar u(\tu_h-\bar u_h)dx\notag\\
  &\hspace{30pt}
  =\int_\O (\tu_h-\bar u_h)\tilde p_h dx+\beta\int_\O \bar u(\tu_h-\bar u_h)dx\notag\\
  &\hspace{30pt}
  =\int_\O (\tilde p_h-\bar p)(\tu_h-\bar u_h)dx+
  \int_\O (\bar p+\beta\bar u)(\tu_h-\bar u_h)dx.\notag
\end{align}
 Moreover, we have
\begin{align*}
  \int_\O (\bar p+\beta\bar u)(\tu_h-\bar u_h)dx
    &=\int_\O (\lambda_a+\lambda_b)(\tu_h-\bar u_h)dx\\
    &=\int_\O \lambda_a(Q_h\bar u-\bar u)dx+\int_\O\lambda_b(Q_h\bar u-\bar u)dx\\
    &\hspace{30pt}+\int_\O \lambda_a(\bar u-a)dx+\int_\O\lambda_b(\bar u-b)dx\\
    &\hspace{50pt}+\int_\O\lambda_a(a-\bar u_h)dx+\int_\O\lambda_b(b-\bar u_h)dx\\
    &\leq\int_\O (\lambda_a-Q_h\lambda_a)(Q_h\bar u-\bar u)dx
    +\int_\O(\lambda_b-Q_h\lambda_b)(Q_h\bar u-\bar u)dx
\end{align*}
 by  \eqref{eq:KKT1}--\eqref{eq:KKT3} and \eqref{eq:tuh}, and hence,
 because of \eqref{eq:MultiplierRegularity}, Remark~\ref{rem:Qhbaru} and \eqref{eq:QhBdd},
\begin{equation}\label{eq:Estimate1}
  \int_\O (\bar p+\beta\bar u)(\tu_h-\bar u_h)dx\leq C_1h^2
\end{equation}
 for the discretization in Section~\ref{subsec:FEM}, and we can
 take $C_1$ to be $0$ for the
 variational discretization in Section~\ref{subsec:Variational}
 because in this case $Q_h\bar u=\bar u$.
\par
 Combining \eqref{eq:Splitting1}--\eqref{eq:Estimate1}, we arrive at the crucial  estimate
\begin{align}\label{eq:SetUp}
   &\sum_{j=1}^J[\bar y(x_j)-\bar y_h(x_j)]^2+\beta\|\bar u-\bar u_h\|_\LT^2\notag\\
    &\hspace{30pt}\leq\sum_{j=1}^J[\bar y(x_j)-\bar y_h(x_j)][\bar y(x_j)-\ty_h(x_j)]
    +\beta\int_\O (\bar u-\bar u_h)(\bar u-\tu_h)dx\\
        &\hspace{60pt}+\int_\O (\tilde p_h-\bar p)(\tu_h-\bar u_h)dx+C_1h^2.\notag
\end{align}
\begin{remark}\label{rem:MeshSize}
  For the finite element discretization in Section~\ref{subsec:FEM},
  the estimate \eqref{eq:SetUp} is valid for sufficiently small $h$, while it is valid
  for any $h$ in the case of the variational discretization in Section~\ref{subsec:Variational}.
\end{remark}
\section{Error Estimates}\label{sec:Errors}
 In view of \eqref{eq:QhBdd}, the key is to estimate the first
 sum and the second integral on the right-hand side of
 \eqref{eq:SetUp}.
\subsection{Estimate for $|\bar y(x_j)-\ty_h(x_j)|$}\label{subsection:KeyEstimate1}
 We begin with the discretization in Section~\ref{subsec:FEM}.
 Let $\ty\in H^1_0(\O)$ be defined by
\begin{equation}\label{eq:ty}
  \int_\O \nabla\ty\cdot\nabla z\,dx=\int_\O \tu_hz\,dx\qquad\forall\,z\in H^1_0(\O).
\end{equation}
It follows from \eqref{eq:ERBdd}, \eqref{eq:Sobolev},
\eqref{eq:State}, \eqref{eq:QhBdd} and  \eqref{eq:ty}
that
\begin{equation}\label{eq:Key1}
  |\bar y(x_j)-\tilde y(x_j)|\leq C\|\bar y-\tilde y\|_{H^2(\O)}
  \leq C\|\bar u-\tu_h\|_\LT\leq Ch.
\end{equation}
\par
 Note that \eqref{eq:tyh} and \eqref{eq:ty} imply
\begin{equation}\label{eq:tyhAndty}
 \ty_h= R_h\ty,
\end{equation}
 where $R_h:H^1_0(\O)\longrightarrow V_h$ is the Ritz projection operator defined by
\begin{equation}\label{eq:Rh}
  \int_\O \nabla(R_h\zeta)\cdot\nabla v_hdx=
  \int_\O \nabla\zeta\cdot\nabla v_hdx\qquad\forall\,\zeta\in H^1_0(\O),
  v_h\in V_h.
\end{equation}
 Therefore we can estimate $|\ty(x_j)-\ty_h(x_j)|$ by the interior
 maximum norm estimate in
 \cite[Corollary~5.1]{SW:1977:IntMaxNorm} or \cite[Theorem~10.1]{Wahlbin:2000:Handbook}.
\par
 Let $B(x_j,\epsilon_j)$ ($1\leq j\leq J$) be an open disc/ball centered at
 $x_j$ with radius $\epsilon_j$
 such that the closed balls $\bar{B}(x_j,\epsilon_j)$ are mutually disjoint
  compact subsets of $\O$.
 It follows from \eqref{eq:tyhAndty}, an interior maximum norm estimate,
 a standard $\LT$ error estimate (cf. \cite{Ciarlet:1978:FEM,BScott:2008:FEM})
 and the convexity of
 $\O$ that
\begin{align}\label{eq:InteriorMaxNorm}
  |\tilde y(x_j)-\tilde y_h(x_j)|&\leq C\big(|\ln h|\inf_{v\in V_h}\|\ty-v\|_{L^\infty(B(x_j,\epsilon_j/4))}
    +\|\ty-\ty_h\|_\LT\big)\\
  &\leq
  C\big(|\ln h|\inf_{v\in V_h}\|\tilde y-v\|_{L^\infty(B(x_j,\epsilon_j/4))}+h^2\|\tu_h\|_\LT).\notag
\end{align}
\par
 One can show by a standard interpolation error estimate (cf. \cite[Theorem~3.15]{Ciarlet:1978:FEM}
 and \cite[Corollary~4.4.7]{BScott:2008:FEM}) and
 an interior $W^{2,p}$ estimate (cf. Lemma~\ref{lem:Interior}) that, for $h$ sufficiently small,
\begin{align}\label{eq:CZ}
  \inf_{v\in V_h}\|\tilde y-v\|_{L^\infty(B(x_j,\epsilon_j/4))}
  &\leq Ch^{2-(d/p)}|\tilde y|_{W^{2,p}(B(x_j,\epsilon_j/2))}\\
  &\leq Ch^{2-(d/p)}\big(\|\tu_h\|_\LT+p\|\tilde u_h\|_{L^\infty(B(x_j,3\epsilon_j/4))}\big)
  \qquad\forall\, p>1.\notag
\end{align}
 Here the positive constant
 $C$ is independent of $h$ and $p$.
\par
 It follows from \eqref{eq:ControlConstraint}, \eqref{eq:tuh}, \eqref{eq:InteriorMaxNorm} and \eqref{eq:CZ}
 (with $p=|\ln h|$) that
\begin{equation}\label{eq:Key2}
   |\tilde y(x_j)-\tilde y_h(x_j)|\leq C_\flat h^2|\ln h|^2.
\end{equation}
\par
 Putting \eqref{eq:Key1} and \eqref{eq:Key2} together, we obtain
\begin{equation}\label{eq:KeyEstimate1}
  |\bar y(x_j)-\tilde y_h(x_j)|\leq C_2h.
\end{equation}
\par
 For the variational discretization in Section~\ref{subsec:Variational},
 we have $\tu_h=\bar u$ that
 implies $\tilde y=\bar y$.  Note that $\bar y$ belongs to $W^{3,s_d}_{loc}(\O)$
 by \eqref{eq:OCRegularity}  and interior elliptic regularity
  (cf. \cite[Theorem~9.19]{GT:2001:EllipticPDE}).  Consequently
 $\bar y$ belongs to $C^2(\O)$ by the Sobolev embedding
 $W^{1,s}_{loc}(\O)\hookrightarrow C(\O)$ for
 $s>d$ (cf. \cite{ADAMS:2003:Sobolev}). It then follows from the
 interior maximum norm estimate that
\begin{equation}\label{eq:VDEstimate}
  |\tilde y(x_j)-\tilde y_h(x_j)|\leq C_\natural h^2|\ln h|.
\end{equation}
\subsection{Estimate for $\d\int_\O (\tilde p_h-\bar p)(\tu_h-\bar u_h)dx$}
\label{subsec:KeyEstimate2}
 Following \cite{Scott:1973:Singular,Casas:1985:Singular} we
  estimate the integral by a duality argument.
 Let $\psi\in H^2(\O)\cap H^1_0(\O)$ be defined by
\begin{equation}\label{eq:psi}
  \int_\O \nabla\psi\cdot\nabla z\,dx=\int_\O (\tu_h-\bar u_h)z\,dx\qquad\forall\,z\in H^1_0(\O).
\end{equation}
\par
 Then we have, by \eqref{eq:Adjoint}, \eqref{eq:tildeph}, \eqref{eq:Rh} and \eqref{eq:psi},
\begin{align}\label{eq:DualityEstimate1}
  \int_\O (\tilde p_h-\bar p)(\tilde u_h-\bar u_h)dx
  &=\int_\O \nabla\psi\cdot\nabla \tilde p_h dx+\int_\O\bar p \Delta\psi\,dx\notag\\
  &=\int_\O \nabla(R_h\psi)\cdot \nabla\tilde p_h dx-\sum_{j\in\cJ}[\bar y(x_j)-d_j]\psi(x_j)\\
  &=\sum_{j\in\cJ}[\bar y(x_j)-d_j][(R_h\psi)(x_j)-\psi(x_j)].\notag
\end{align}
\par
 Again, by \eqref{eq:psi} and the interior maximum norm estimate, we have
\begin{equation}\label{eq:DualityEstimate2}
  |(R_h\psi)(x_j)-\psi(x_j)|\leq C h^2|\ln h|^2\big(
  \|\tu_h-\bar u_h\|_{L^\infty(B(x_j,3\epsilon_j/4))}+\|\tu_h-\bar u_h\|_\LT\big),
\end{equation}
 and hence, in view of \eqref{eq:ControlConstraint}, \eqref{eq:DControl}, \eqref{eq:tuh} and \eqref{eq:DualityEstimate1},
\begin{equation}\label{eq:DualityEstimate3}
  \int_\O (\tilde p_h-\bar p)(\tu_h-\bar u_h)dx\leq C_3h^2|\ln h|^2.
\end{equation}
%
\subsection{Error Estimate for the Finite Element Discretization in Section~\ref{subsec:FEM}}
\label{subsec:ErrorFEM}
Putting \eqref{eq:QhBdd}, \eqref{eq:SetUp}, \eqref{eq:KeyEstimate1}
and \eqref{eq:DualityEstimate3}
together, we arrive at the estimate
\begin{align}\label{eq:FEMError1}
   &\sum_{j=1}^J[\bar y(x_j)-\bar y_h(x_j)]^2+\beta\|\bar u-\bar u_h\|_\LT^2\\
    &\hspace{30pt}\leq C_2h\sum_{j=1}^J|\bar y(x_j)-\bar y_h(x_j)|
    +\beta C_\sharp h\|\bar u-\bar u_h\|_\LT
        +C_3h^2|\ln h|^2+C_1h^2.\notag
\end{align}
\par
 It then follows from \eqref{eq:FEMError1} and Young's inequality
\begin{equation}\label{eq:Young}
   xy\leq \frac{x^2}{2\epsilon}+\frac{\epsilon y^2}{2}
   \qquad\forall\,x,y\geq0 \quad\text{and}\quad \epsilon>0
\end{equation}
 that
\begin{equation}\label{eq:FEMError}
  \sum_{j=1}^J[\bar y(x_j)-\bar y_h(x_j)]^2+\beta\|\bar u-\bar u_h\|_\LT^2\leq Ch^2|\ln h|^2.
\end{equation}
%
\subsection{Error Estimate for the Variational Discretization in Section~\ref{subsec:Variational}}
\label{subsec:VD}
 Since $\tu_h=Q_h\bar u=\bar u$, $\ty=\bar y$ and $C_1=0$, we obtain from
 \eqref{eq:SetUp}, \eqref{eq:Key2} and \eqref{eq:DualityEstimate3}
 the estimate
\begin{align}\label{eq:VDError1}
   &\sum_{j=1}^J[\bar y(x_j)-\bar y_h(x_j)]^2+\beta\|\bar u-\bar u_h\|_\LT^2\\
  &\hspace{50pt}
    \leq C_\natural h^2|\ln h|\sum_{j=1}^J|\bar y(x_j)-\bar y_h(x_j)|
           +C_3h^2|\ln h|^2.\notag
\end{align}
\par
 It follows from \eqref{eq:Young} and \eqref{eq:VDError1} that
\begin{equation}\label{eq:VDError}
  \sum_{j=1}^J[\bar y(x_j)-\bar y_h(x_j)]^2+\beta\|\bar u-\bar u_h\|_\LT^2\leq Ch^2|\ln h|^2.
\end{equation}

%

\subsection{Improved Error Estimates in Two Dimensions}\label{subsec:2DFEM}
 Both estimates \eqref{eq:FEMError} and \eqref{eq:VDError} can be improved when $d=2$.
%
\subsubsection{Finite Element Discretization in Section~\ref{subsec:FEM}}
\label{subsubsec:FEMImprovement}
\par
 In this case we can replace the estimate \eqref{eq:DualityEstimate2}
 by the pointwise estimate
\begin{equation}\label{eq:SubOptimal}
  |(R_h\psi)(x_j)-\psi(x_j)|\leq \|R_h\psi-\psi\|_{L^\infty(\O)}\leq Ch|\psi|_{H^2(\O)}
\end{equation}
 (cf. \cite[Exercise~3.3.2]{Ciarlet:1978:FEM}) so that
\begin{align}\label{eq:FEMBetterDuality}
  \int_\O (\tilde p_h-\bar p)(\tu_h-\bar u_h)dx&\leq C_4h\|\tu_h-\Bar u_h\|_\LT\\
  &\leq C_4h\big(\|Q_h\bar u-\bar u\|_\LT+\|\bar u-\bar u_h\|_\LT)\notag
\end{align}
 by \eqref{eq:ERBdd}, \eqref{eq:tuh}, \eqref{eq:psi},
 \eqref{eq:DualityEstimate1} and \eqref{eq:SubOptimal}.
\par
 Taking  \eqref{eq:QhBdd} and \eqref{eq:FEMBetterDuality} into account,
 the estimate \eqref{eq:FEMError1} becomes
\begin{align}\label{eq:2DFEMError1}
   &\sum_{j=1}^J[\bar y(x_j)-\bar y_h(x_j)]^2+\beta\|\bar u-\bar u_h\|_\LT^2\\
    &\hspace{30pt}\leq C_2h\sum_{j=1}^J|\bar y(x_j)-\bar y_h(x_j)|
    +(\beta C_\sharp+C_4) h\|\bar u-\bar u_h\|_\LT
       +(C_4C_\sharp+C_1 )h^2.\notag
\end{align}
 It follows from \eqref{eq:Young} and \eqref{eq:2DFEMError1} that
\begin{equation}\label{eq:2DFEMBetterEstimate}
  \sum_{j=1}^J[\bar y(x_j)-\bar y_h(x_j)]^2+
  \beta\|\bar u-\bar u_h\|_\LT^2\leq C(1+\beta^{-1})h^2.
\end{equation}
\subsubsection{Variational Discretization in Section~\ref{subsec:Variational}}
\label{subsubsec:VariationalImprovement}
 In this case
 the relation \eqref{eq:DControl} becomes
\begin{equation}\label{eq:VDControl}
   \bar u_h=\max(a,\min(b,-(\bar p_h/\beta)).
\end{equation}
 Note that $\bar y_h(x_j)$ converges to $\bar y(x_j)$ as $h$ decreases to $0$ by
 \eqref{eq:VDError}.  At a tracking point $x_j$ where $\bar y(x_j)-d_j\neq0$,
 the relation \eqref{eq:DAdjoint}, the growth
  of the discrete Green's function
 associated with a tracking point (cf. Lemma~\ref{lem:DG})
 and \eqref{eq:VDControl} imply that $\bar u_h$ equals the same constant as $\bar u$ in a neighborhood
 $B(x_j,\delta_j)$ for $h$ sufficiently small.
 Therefore, by \eqref{eq:Sobolev} and the Poisson integral formula
  for harmonic functions  (cf. \cite[Equation~2.26]{GT:2001:EllipticPDE}),
  the function $\psi$ defined by \eqref{eq:psi} satisfies
\begin{equation*}
  \|\psi\|_{W^{2,2}B(x_j,\delta_j/2)}\leq C\|\bar u-\bar u_h\|_\LT
\end{equation*}
 because $\tu_h=\bar u$,
 and
 the estimate \eqref{eq:DualityEstimate2} can be replaced by
\begin{equation}\label{eq:DualityEstimate4}
   |(R_h\psi)(x_j)-\psi(x_j)|\leq C h^2|\ln h|\|\bar u-\bar u_h\|_\LT.
\end{equation}
\par
 We deduce from \eqref{eq:DualityEstimate1} and \eqref{eq:DualityEstimate4} that
  \eqref{eq:DualityEstimate3} can be improved to
\begin{equation}\label{eq:NewAdjointEstimate}
  \int_\O (\tilde p_h-\bar p)(\bar u-\bar u_h)dx\leq C_5h^2|\ln h|\|\bar u-\bar u_h\|_\LT.
\end{equation}
\par
 In view of \eqref{eq:NewAdjointEstimate},
  the estimate \eqref{eq:VDError1}  becomes
\begin{align*}
   &\sum_{j=1}^J[\bar y(x_j)-\bar y_h(x_j)]^2+\beta\|\bar u-\bar u_h\|_\LT^2\\
    &\hspace{30pt}\leq C_\natural h^2|\ln h|\sum_{j=1}^J|\bar y(x_j)-\bar y_h(x_j)|
           +C_5h^2|\ln h|\|\bar u-\bar u_h\|_\LT,\notag
\end{align*}
which together with \eqref{eq:Young} implies
\begin{equation}\label{eq:2VDError}
  \sum_{j=1}^J[\bar y(x_j)-\bar y_h(x_j)]^2+
  \beta\|\bar u-\bar u_h\|_\LT^2\leq C(1+\beta^{-1}) h^4|\ln h|^2.
\end{equation}
\section{Concluding Remarks} \label{sec:Conclusion}
 We have developed a concise convergence analysis that is based on the
  estimate \eqref{eq:SetUp} where two crucial terms
 for the error estimate are identified.
 Standard finite element techniques combined with classical results
  in elliptic partial differential equations
  then enable us to obtain the error estimates
 \eqref{eq:FEMError}, \eqref{eq:VDError}, \eqref{eq:2DFEMBetterEstimate}
 and \eqref{eq:2VDError}
 for the finite element and variational discretizations in two and
 three dimensions that appeared in
 \cite{CGY:2015:Point,BDE:2016:Points,AOS:2018:Weighted,BMV:2019:Point,VM:2025:EOCP}.
\par
 We have used the Poisson equation constraint \eqref{eq:StateEquation} for simplicity.
 But the results also hold for
 a general elliptic equation constraint
\begin{equation}\label{eq:General}
  \int_\O \Big[\sum_{i,j=1}^d a_{ij}(x)\frac{\p y}{\p x_i}
    \frac{\p z}{\p x_j}+\sum_{i=1}^da_i(x)
    \frac{\p y}{\p x_i}z+a(x)yz\Big]dx=\int_\O uz\,dx
    \qquad\forall\,z\in H^1_0(\O),
\end{equation}
 provided the coefficients are sufficiently smooth and the problem
 \eqref{eq:General} is well-posed.
\par
 One can also apply the finite element discretization to the optimal control
 problem \eqref{eq:OCP}--\eqref{eq:ControlConstraint}
 where the constants $a$ and $b$ are replaced by two functions $a(x)$ and $b(x)$,
 and \eqref{eq:DControlConstraint} is replaced by
   $Q_ha\leq u_h\leq Q_hb$.
 An examination of the arguments shows that the error estimates in Section~\ref{subsec:ErrorFEM} and
 Section~\ref{subsubsec:FEMImprovement} are valid
  under the assumptions that $a(x),b(x)\in H^2(\O)$ and $a(x)< b(x)$ in $\O$.
%
\appendix

\section{An Interior $W^{2,p}$ Estimate}\label{append:Interior}
 Let $y\in H^1_0(\O)$ and $u\in L^\infty(\O)$ satisfy the
 constraint \eqref{eq:StateEquation}.
 Given a point $x_*\in\O$, let $\epsilon_*$ be a small number
 such that the closed ball
 $\bar B(x_*,\epsilon_*)$ is a subset of $\O$.
\begin{lemma}\label{lem:Interior}
 We have
\begin{equation}\label{eq:InteriorEstimate}
   |y|_{W^{2,p}(B(x_*,\epsilon_*/2))}\leq
   C\big(\|u\|_\LT+p\|u\|_{L^\infty(B(x_*,3\epsilon_*/4))}\big)
   \qquad\text{for all} \quad 1<p<\infty,
\end{equation}
 where the positive constant $C$ is independent of $p$.
\end{lemma}
\begin{proof}
  Let $\rho$ be a $C^\infty$ function such that $\rho=1$
   on $B(x_*,3\epsilon_*/4)$ and
  $\rho(x)=0$ if $|x-x_*|>\epsilon_*$.  Then we have
   (cf. \cite[Theorem~9.9]{GT:2001:EllipticPDE})
\begin{equation}\label{eq:Interior1}
  \rho(x) y(x)=\int_\O \Gamma(x-z)\Delta_z(\rho(z) y(z))dz,
\end{equation}
 where $\G(x)$ is the fundamental solution of the Laplace operator in \eqref{eq:FS}.
\par
 We derive from \eqref{eq:StateEquation}, \eqref{eq:Interior1} that,
 for $|\alpha|=2$ and $x\in B(x_*,\epsilon_*/2)$,
\begin{align}\label{eq:Interior2}
  \frac{\p^\alpha y}{\p x^\alpha}(x)
  &=-\int_{B(x_*,3\epsilon_*/4)}\frac{\p^\alpha}{\p x^\alpha}\G(x-z)u(z)dz\\
  &\hspace{30pt}
 + \int_{\O\setminus B(x_*,3\epsilon_*/4)}\frac{\p^\alpha}{\p x^\alpha}\G(x-z)\Delta_z(\rho(z)y(z))dz
  =\phi_1(x)+\phi_2(x).\notag
\end{align}
\par
 It follows from the Calderon-Zigmund inequality that
 $\phi_1$ belongs to $L^p(B(x_*,\epsilon_*/2))$ and
\begin{equation}\label{eq:Interior3}
  \|\phi_1\|_{L^p(B(x_*,\epsilon_*/2))}\leq Cp\|u\|_{L^p(B(x_*,3\epsilon_*/4))}
  \qquad\forall\,p>1
\end{equation}
(cf. \cite[Theorem~9.8 and Theorem~9.9]{GT:2001:EllipticPDE}).
\par
 On the other hand, since there is a gap between $B(x_*,\epsilon_*/2)$
 and $\O\setminus B(x_*,3\epsilon_*/4)$,
 we can conclude that
\begin{equation}\label{eq:Interior4}
  \|\phi_2\|_{L^p(B(x_*,\epsilon_*/2))}\leq
  C\|\Delta(\rho y)\|_{L^2(\O\setminus B(x_*,3\epsilon_*/4))}
  \leq C\|u\|_{L_2(\O)},
\end{equation}
 where we have also used elliptic regularity (cf. Remark~\ref{rem:ER}).
\par
 The estimate \eqref{eq:InteriorEstimate} follows from
 \eqref{eq:Interior2}--\eqref{eq:Interior4}.
\end{proof}
%
\section{The Growth of Discrete Green's Functions in Two Dimensions}\label{append:DG}
 Let $x_*$ be a point in $\O$ and the discrete Green's function $G_{*,h}\in V_h$ be defined by
\begin{equation}\label{eq:DG1}
  \int_\O \nabla G_{*,h}\cdot\nabla v_h dx=v_h(x_*) \qquad\forall\,v_h\in V_h.
\end{equation}
\begin{lemma}\label{lem:DG}
  Given $M>0$, there exists $\epsilon_*>0$ independent of $h$
  such that $\bar B(x_*,\epsilon_*)\subset \O$ and
\begin{equation}\label{eq:DG2}
  G_{*,h}(x)>M \qquad\forall\,x\in B(x_*,\epsilon_*),
\end{equation}
 provided $h$ is sufficiently small.
\end{lemma}
\begin{proof}
Let $x_*$ belong to a triangle $T_*\in\cT_h$.  There exists a function $\delta_{T_*}\in\LT$
with the following properties.
\begin{equation}\label{eq:Support}
\text{$\delta_{T_*}$ vanishes outside $\tilde T_*$},
\end{equation}
 where $\tilde T_*\subset T_*$ is the triangle bounded by the
 lines $\lambda_1=1/6$, $\lambda_2=1/6$  and
 $\lambda_3=1/6$, and $\lambda_1,\lambda_2,\lambda_3$ are
  the barycentric coordinates for $T_*$.
 Moreover, we have
\begin{align}
  \int_{T_*}\delta_{T_*}v\,dx&=v(x_*) \qquad\forall\,v\in P_1(T_*),\label{eq:DG3}\\
  \|\delta_{T_*}\|_{L^1(T_*)}&\leq C_1,\label{eq:DG4}\\
  \|\delta_{T_*}\|_{L^2(T_*)}&\leq C_2h^{-1},\label{eq:DG5}\\
  \|\delta_{T_*}\|_{L^\infty(T_*)}&\leq C_3 h^{-2},\label{eq:DGnew1}
\end{align}
 where the positive constants $C_1$, $C_2$ and $C_3$ only depend on the
 shape regularity of $T_*$.
 The construction of $\delta_{T_*}$ can be found for example in
 \cite[Theorem~1]{Scott:1973:Singular} or \cite[Section~A.5]{SW:1995:InteriorMax}.
\par
 Let $G_*\in H^1_0(\O)$ be defined by
\begin{equation}\label{eq:DG6}
  \int_\O \nabla G_*\cdot\nabla v\,dx=
  \int_\O \delta_{T_*}v\,dx\qquad\forall\,v\in H^1_0(\O).
\end{equation}
 Then $G_*\in H^2(\O)$ by elliptic regularity,
\begin{equation}\label{eq:LaplaceGh}
 -\Delta G_*=\delta_{T_*},
 \end{equation}
and, in view of
 \eqref{eq:DG1}, \eqref{eq:DG3} and \eqref{eq:DG6}, $G_{*,h}$ is
  the Ritz projection of $G_*$.
\par
 Let the function $N_*$ be defined by
\begin{equation}\label{eq:DG7}
  N_*(x)=-\int_{\tilde T_*}\frac{\ln|x-y|}{2\pi}\delta_{T_*}(y)dy.
\end{equation}
 Then we have $N_*\in H^2(\O)$ and
\begin{equation}\label{eq:NewtonPotential}
  \Delta N_*=-\delta_{T_*}
\end{equation}
(cf. \cite[Theorem~9.9]{GT:2001:EllipticPDE}).
\par
 Let $\rho$ be a function in $C^\infty_c(\O)$ such that
\begin{equation}\label{eq:rho}
 \text{$\rho=1$ on $B(x_*,\gamma)$ for some $\gamma>0$.}
\end{equation}
 The function $\rho N_*$ belongs to $H^2(\O)\cap H^1_0(\O)$,
 and by \eqref{eq:DG4}
 and
 \eqref{eq:LaplaceGh}--\eqref{eq:rho},
\begin{equation}\label{eq:DG8}
  \| \Delta(\rho N_*-G_*)\|_{L^\infty(\O)}=\|\Delta(\rho N_*)+\delta_{T_*}\|_{L^\infty(\O)}
  =\|\Delta(\rho N_*)\|_{L^\infty(\O\setminus B(x_*,\gamma))}\leq C_4,
\end{equation}
 provided that $h$ is sufficiently small so that $T_*$ is a subset of $B(x_*,\gamma/2)$.
 Here the positive constant $C_4$ depends only on $\gamma$ and the
  constant $C_1$ in \eqref{eq:DG4}.
\par
 It follows from \eqref{eq:ERBdd}, \eqref{eq:Sobolev}, \eqref{eq:SubOptimal}
  and \eqref{eq:DG8} that
\begin{align}\label{eq:DG9}
  \|R_h(\rho N_*)-G_{*,h}\|_{L^\infty(\O)}&
  =\|R_h(\rho N_*-G_{*})\|_{L^\infty(\O)}\notag\\
  &\leq \|R_h(\rho N_*-G_{*})-(\rho N_*-G_{*})\|_{L^\infty(\O)}
  +\|\rho N_*-G_{*}\|_{L^\infty(\O)}\\
  &\leq C\|\rho N_*-G_{*}\|_{H^2(\O)}\leq C_5,\notag
\end{align}
 where the positive constant $C_5$ only depends on $C_4$ and
  the constants in \eqref{eq:ERBdd} and \eqref{eq:SubOptimal}.
\par
 Furthermore we find, from \eqref{eq:DG5}, \eqref{eq:LaplaceGh} and \eqref{eq:DG8},
\begin{equation*}
 \|\Delta(\rho N_*)\|_\LT\leq \|\Delta(\rho N_*-G_*)\|_\LT+\|\Delta G_*\|_\LT\leq C_6 h^{-1},
\end{equation*}
 where the positive constant $C_6$ depends only on $C_2$ and $C_4$.
 Therefore we have
\begin{equation}\label{eq:DG11}
  \|\rho N_*-R_h(\rho N_*)\|_{L^\infty(\O)}\leq C_7,
\end{equation}
 where the positive constant $C_7$ depends only on $C_6$ and the constants in
 \eqref{eq:ERBdd} and \eqref{eq:SubOptimal}.
\par
 We see from \eqref{eq:DG9} and \eqref{eq:DG11} that it only remains to show
 $\rho N_*$ is large near $x_*$.
\par
 For $x\in T_*$, we can use \eqref{eq:DG3} and a change of variables to write
\begin{align}\label{eq:DGnew2}
  \rho(x)N(x)&=-\int_{T_*}\frac{\ln |x-y|}{2\pi}\delta_{T_*}(y)dy\notag\\
  &=-\frac{\ln h}{2\pi}-\int_{T_*}\frac{\ln |(1/h)(x-y)|}{2\pi}\delta_{T_*}(y)dy\\
   &=-\frac{\ln h}{2\pi}-\int_{\hat T_*}\frac{\ln|\hat x-\hat y|}{2\pi}\delta_{T_*}(h\hat y) h^2d\hat y,\notag
\end{align}
 where $\hat x=(1/h)x$ and $\hat T_*=(1/h)T_*$.  In view of \eqref{eq:DGnew1}, we have
\begin{equation}\label{eq:DGnew3}
  \Big|\int_{\hat T_*}\frac{\ln|\hat x-\hat y|}{2\pi}\delta_{T_*}(h\hat y) h^2d\hat y,\Big|
  \leq C_8,
\end{equation}
 where the positive constant only depends on $C_3$.
\par
 On the other hand,
 for $x\in B(x_*,\gamma)\setminus T_*$, we have, by \eqref{eq:DG3}, \eqref{eq:DG7} and \eqref{eq:rho},
\begin{equation}\label{eq:DG12}
  \rho(x) N_*(x)=-\int_{\tilde T_*}\frac{\ln |x-y|}{2\pi}\delta_{T_*}(y)dy
  =-\frac{\ln |x-c_*|}{2\pi}+
  \frac{1}{2\pi}\int_{\tilde T_*}\delta_{T_*}(y)\ln\frac{|x-c_*|}{|x-y|}dy,
\end{equation}
 where $c_*$ is the center of $T_*$.
\par
 Note that \eqref{eq:Support} and the quasi-uniformity of $\cT_h$
  imply $|x-c_*|$ and $|x-y|$ are
 comparable for $y\in\tilde T_*$.  It then follows from \eqref{eq:DG4} that
\begin{equation}\label{eq:DG13}
  \frac{1}{2\pi}\Big|\int_{\tilde T_*}\delta_{T_*}(y)\ln\frac{|x-c_*|}{|x-y|}
  dy\Big|\leq C_9,
\end{equation}
 where the positive constant $C_9$ is independent of $h$.
\par
 Putting \eqref{eq:DG9}--\eqref{eq:DG13} together, we see that,
 for $x\in B(x_*,\gamma)$,
\begin{align*}
  G_{*,h}(x)&\geq [R_h(\rho N_*)](x)-C_5\\
   &\geq \rho(x)N_*(x)-C_7-C_5\\
   &\geq
   \begin{cases}
    \d -\frac{\ln h}{2\pi}-C_8-C_7-C_5&\qquad\text{if $x\in T_*$},\\[10pt]
     \d -\frac{\ln|x-c_*|}{2\pi}-C_9-C_7-C_5&\qquad\text{if $x\in B(x_*,\gamma)\setminus T_*$}.
   \end{cases}
\end{align*}
\par
 Let $\epsilon_*<\gamma/2$ be a small positive number such that
\begin{equation*}
 -\frac{\ln|x-c_*|}{2\pi}-C_9-C_7-C_5> -\frac{\ln (2\epsilon_*)}{2\pi}-C_9-C_7-C_5>M\qquad
 \text{if $|x-c_*|<2\epsilon_*$}.
\end{equation*}
 Then  \eqref{eq:DG2} is valid if $h$ is sufficiently small so that $|x-x_*|<\epsilon_*$ implies
 $|x-c_*|<2\epsilon_*$ and also
\begin{equation*}
   -\frac{\ln h}{2\pi}-C_8-C_7-C_5>M.
\end{equation*}
\end{proof}
\begin{remark}\label{rem:Dimitri}
  Another proof of Lemma~\ref{lem:DG} based on the results in
  \cite{LP:2017:Positivity} can be found in \cite{BMV:2019:Point}.
\end{remark}
\begin{remark}\label{rem:Dimensions}
  The arguments for Lemma~\ref{lem:DG} do not work in three
  dimensions because \eqref{eq:DG5}
  becomes
    $\|\delta_{T_*}\|_\LT\leq C_2h^{-\frac32}$
  and then \eqref{eq:DG11} becomes
   $\|\rho N_*-R_h(\rho N_*)\|_{L^\infty(\O)}\leq Ch^{-\frac12}$.
\end{remark}

\end{document}